\documentclass[journal]{IEEEtran}

\usepackage{algorithm}
\usepackage{algorithmic}
\usepackage[backend=biber,style=ieee,sorting=none, doi=false]{biblatex}

\usepackage{amsthm}

\usepackage{amsmath}                        

\usepackage{graphicx}                       
\usepackage{enumitem}                       

\usepackage{amssymb}                        
\usepackage{mathtools}                      

\MakeRobust{\eqref}

\newtheorem{theorem}{Theorem}
\newtheorem{lemma}{{Lemma}}

\newtheorem{definition}{ {Definition}}

\newtheorem{remark}{ {Remark}}
\newtheorem{assumption}{ {Assumption}}

\usepackage[usenames, dvipsnames, svgnames, table]{xcolor}

\newcommand{\norm}[1]{\left\|#1\right\|}

\usepackage{tikz} 
\usetikzlibrary{matrix}
\usepackage[plainpages=false, pdfpagelabels]{hyperref} 
	\hypersetup{
		colorlinks   = true,
		citecolor    = RoyalBlue,
		linkcolor    = RubineRed,
		urlcolor     = black
		}
\usepackage{threeparttable}
\usepackage{booktabs}
\usepackage{cleveref}
\usepackage{bm}
\usepackage{svg}
\usepackage{subcaption}
\usepackage{soul}
\usepackage{adjustbox}
\usepackage{tikz}
\usetikzlibrary{calc}
\definecolor{ultramarine}{RGB}{31, 119, 180}

\newcommand{\submin}{\text{min}}
\newcommand{\submax}{\text{max}}
\newcommand{\subinit}{\text{init}}
\newcommand\toggleFScomm{1}

\newcommand{\MR}{\mathbb{R}}

\newcommand{\Rfunc}[3]{{#1} \colon \mathbb{R}^{#2} \to \mathbb{R}^{#3}}

\newcommand{\intset}[2]{\left[#1,#2\right]_{\mathbb{N}}}
\newcommand{\vnorm}[1]{\|#1\|}
\newcommand{\blue}[1]{\textcolor{black}{#1}}

\ifCLASSINFOpdf
\else
\fi

\begin{document}
%
\title{Discrete-Time Lossless Convexification for Pointing Constraints}
\author{Dayou Luo\textsuperscript{\footnotemark{1}}, Fabio Spada\textsuperscript{\footnotemark{2}}, Beh\c cet A\c c{\i}kme\c se\textsuperscript{\footnotemark{3}} \vspace{-0.5cm}} 

\maketitle
\footnotetext[1]{Ph.D. Candidate, Department of Applied Mathematics, University of Washington, Seattle, WA 98195.
{\tt\small dayoul@uw.edu}}
\footnotetext[2]{Ph.D. Student, William E. Boeing Department of Aeronautics
and Astronautics, University of Washington, Seattle, WA 98195.
        {\tt\small fspada46@uw.edu}}
\footnotetext[3]{Professor, William E. Boeing Department of Aeronautics
and Astronautics, University of Washington, Seattle, WA 98195.
        {\tt\small behcet@uw.edu}}

\begin{abstract}

Discrete-Time Lossless Convexification (DT-LCvx) formulates a convex relaxation for a specific class of discrete-time non-convex optimal control problems. It establishes sufficient conditions under which the solution of the relaxed problem satisfies the original non-convex constraints at specified time grid points. Furthermore, it provides an upper bound on the number of time grid points where these sufficient conditions may not hold, and thus the original constraints could be violated. This paper extends DT-LCvx to problems with control pointing constraints. Additionally, it introduces a novel DT-LCvx formulation for mixed-integer optimal control problems in which the control is either inactive or constrained within an annular sector. Such formulation broadens the feasible space for problems with pointing constraints. A numerical example is provided to illustrate its application.

\end{abstract}


\section{Introduction}
\if\toggleFScomm1
\IEEEPARstart{C}{onvex} \blue{optimization has been a central tool in trajectory optimization for decades; furthermore, the last two decades have witnessed advancements of convex optimization techniques as well in the domains of guidance and control of autonomous systems \cite{malyuta2022convex, liu2017survey, wang2024survey};}
\else
\IEEEPARstart{I}{n the} most recent years, convex optimization has become the capstone of computational guidance and control strategies for autonomous systems\cite{malyuta2022convex}; \fi
convergence to globally optimal points and polynomial-time complexity constitute its strengths \cite{Nesterov1994}. In this context, a few non-convex Optimal Control Problems (OCPs) can be \textit{relaxed} to provably-equivalent convex OCPs by substituting the non-convex constraints with their convex counterparts. The relaxed problem can then be solved to global optimality while preserving the satisfaction of the original non-convex constraints. This relaxation technique is therefore referred to as Lossless Convexification (LCvx).

LCvx has been successfully applied to\if\toggleFScomm1
\blue{compute optimal constrained trajectories for autonomous landing vehicles, subject to}
\fi annular control constraints \cite{acikmese2007convex, blackmore2010, kunhippurayil2021lossless}, pointing constraints \cite{acikmese2013} and linear and quadratic state constraints \cite{harris2013lossless, harris2014lossless}.  
\blue{Such constraints are often dictated by safety considerations: for example, \textit{pointing constraints} are imposed to limit vehicle's orientation.}
\blue{Moreover, LCvx has been flight-proven \cite{acikmese2013flight} and it is used as a baseline for Moon landing applications \cite{ Shaffer2024-yx}.} \blue{Additional studies have proposed LCvx for orbital rendezvous operations \cite{Lu2013-fb}, and have suggested approaches to alleviate the hypotheses granting the validity of LCvx \cite{Yang2024-rp}}. 
More recent studies tackle non-convexities of mixed-integer (MI) nature and allow to constrain controls to semi-continuous sets \cite{malyuta2020} or to quantize controls \blue{for different endo-atmospheric and exo-atmospheric vehicles} \cite{Harris2021, Kazu2023, yang2024}.
Relaxation may be valid even in presence of state constraints \cite{harris2013lossless, Woodford2022-cr}. 
\blue{LCvx is further used for constrained quadrotor trajectory optimization: in this work, we consider a similar formulation as the one in \cite{Szmuk2017-ve}.} The interested reader is referred to \cite{malyuta2022convex} for an analysis of the status of LCvx. 

\blue{The  previously mentioned examples establish the equivalence between the original continuous-time (CT) non-convex OCP and its convex relaxation. However, this equivalence breaks down when the convex CT-OCP is \textit{discretized} into a discrete-time (DT) OCP using zero-order hold for practical computation; this discretization introduces a set of \textit{time grid points} that define intervals over which the control is held constant; the states and controls at these points form the solution of the DT-OCP.
The solution of the DT convex relaxation may violate the original non-convex constraints, indicating that the two discrete-time optimization problems are no longer equivalent. As a result, CT-LCvx (LCvx for CT-OCP) does not provide theoretical guarantees for its practical implementation i.e. the DT-LCvx (LCvx for DT-OCP) , and a new theoretical framework that directly addresses the discrete-time formulation is required.}

To address this issue, \cite{Luo2024-cy} provides a comprehensive analysis of the relationship between non-convex and convex DT-OCPs by proposing the Discrete-Time LCvx (DT-LCvx) method. DT-LCvx works directly on the DT non-convex problem and its convex relaxation. It offers a tight upper bound on the number of time grid points where the solution of the convex relaxation violates the original non-convex constraints. Unlike the CT version, DT-LCvx uses finite-dimensional tools such as convex analysis and linear algebra to establish its theoretical results. This distinction provides a new perspective on understanding the LCvx technique and opens up possibilities for designing novel relaxations.

\blue{\textbf{Contributions: } This work extends DT-LCvx by incorporating results from~\cite{Luo2024-cy} to address practical optimal control problems (OCPs) in linear time-invariant systems with classic control pointing and semi-continuous magnitude constraints. The pointing constraint — used to ensure safe vehicle orientation during maneuvers~\cite{acikmese2013}—is considered here in a formulation that has been experimentally validated~\cite{acikmese2013flight}. We establish theoretical guarantees for applying DT-LCvx under such pointing constraints, proving that the number of violating grid points along the trajectory is upper bounded by $n_x - 1$.}

\blue{The semi-continuous magnitude constitutes a requirement of mixed-integer structure, by which control inputs are either zero or exceed a nonzero threshold. This captures \textit{dual-mode control} encountered in, for example, on-off actuators~\cite{Yoshimura2012-to}, minimum-impulse bit systems~\cite{Malyuta2023-bh}; additional systems make use of such type of control \cite{malyuta2020, Harris2021, yang2024}. We provide theoretical guarantees similar to the pointing constraints cases, with an upper bound $2n_x - 2$, when the control input lies in $\mathbb{R}^2$. In higher dimensions, although theoretical guarantees remain unavailable, numerical results indicate that DT-LCvx still provides an effective heuristic.}


\textbf{Notation:}
The set of integers $\{a, a+1, \ldots, b\}$ is denoted by $[a,b]_{\mathbb{N}}$. $\mathbb{R}$, $\MR^n$ and $\mathbb{R}^{n \times m}$ are respectively the set of real numbers, the set of $n-$dimensional vectors, and the set of real $n$ by $m$ matrices. The cross product of two vectors is denoted as $\mathbf{a} \times \mathbf{b}$. The Euclidean norm of a vector by $\|\cdot\|$. Horizontal concatenation of matrices is denoted by $[A_1, A_2, \ldots, A_n]$. The Jacobian of a function $H$ is denoted by $\nabla H$. We denote the identity matrix by $\bm{I}$ and the zero matrix by $\bm{0}$. Given a convex set $\mathcal{S}$ and $z\in\mathcal{S}$, then $N_\mathcal{S}(z)$ denotes the convex cone normal to $\mathcal{S}$ at $z$ \cite[Chapter 2.1]{borwein2006convex}.

\section{Problem formulation}
\label{sec:2}
%
%
%
%

This paper addresses two types of lossless convexification for non-convex optimal control problems. The first problem features classical pointing constraints from \cite{acikmese2013}: \vspace{0.2cm}

\textbf{Problem 1 (P1)}:
\begin{equation*}
\label{equ:original_LCvx}
\begin{array}{ll}
\underset{x,u}{\text{minimize}}& \quad \;m\left(x_{N+1}\right) + \sum_{i=1}^N l_i\left(\norm{u_i}\right) \\
 \text{subject to} & \left|\begin{array}{l} \left.\begin{array}{l}x_{i+1} = A x_{i} + B u_i + z_i\\
\rho_\submin\leq \norm{u_i} \leq \rho_\submax\\
\xi ^{\top} u_i \geq \gamma\left\|u_i\right\| \end{array}\right| \; \forall\,i \in \intset 1 N \\[3.0ex]
\;\; x_1 = x_\subinit, \quad G\left(x_{N+1}\right) = 0 \end{array} \right.
\end{array}  
\end{equation*}
with the following solution  variables:
\begin{equation*}
\begin{aligned}
x = & \{x_1, x_2, \cdots, x_{N+1}\}, \  x_i \in \mathbb{R}^{n_x} \text{ for } i \in \intset 1 {N+1}\\
u = & \{u_1, u_2, \cdots, u_N\},\  u_i \  \in \mathbb{R}^{n_u} \text{ for } i \in \intset 1 N\\
\end{aligned}
\end{equation*}
where $x_i$, $u_i$ and $z_i\in \MR ^{n_x}$ are, respectively, the state, the control input, and the external force at time step $i$. The cost functions $\Rfunc{m}{n_x}{}$ and $\Rfunc{l_i}{}{}$ for $i \in \intset {1}{N}$ are convex $C^1$ smooth functions, defining the terminal state cost and the stepwise control cost.  $\Rfunc{G}{n_x}{n_G}$ is an affine map characterizing the boundary condition of the state trajectory. The dynamics of the system is defined by constant matrices $A \in \mathbb{R}^{n_x \times n_x}$ and $B \in \mathbb{R}^{n_x \times n_u}$, with a fixed initial state $x_{\text{init}} \in \mathbb{R}^{n_x}$. The constants $\rho_{\text{min}}, \rho_{\text{max}} \in \mathbb{R}$ bound the control input norm. The pointing constraint $\xi ^{\top} u_i \geq \gamma\left\|u_i\right\|$ restricts the angle between $u_i$ and a fixed unit vector $\xi$, with constant $\gamma >0 $. Furthermore, $l_i$ are monotonically increasing functions on the interval $[\rho_{\min}, \infty)$.

The second problem allows zero control input, as follows.  \vspace{0.2cm}

\textbf{Problem 2 (P2)}:
\label{equ:original_LCvx_mixinteger}
\begin{equation*}
\begin{array}{ll}
\underset{x,u,\theta}{\text{minimize}}& \quad \; m\left(x_{N+1}\right) + \sum_{i=1}^N l_i\left(\max(\vnorm{u_i}, \rho_\submin)\right) \\
 \text{subject to} \quad & \left| \begin{array}{l} \left. \begin{array}{l} x_{i+1} = A x_{i} + B u_i + z_i\\
\theta_i \rho_\submin\leq \vnorm{u_i} \leq  \theta_i \rho_\submax\\
\xi ^{\top} u_i \geq \gamma\left\|u_i\right\|\\
\theta_i \in \{0,1\}\ ,u_i \in \mathbb{R}^2\end{array}\right|\;\forall\,i \in \intset 1 {N} \\[4.5ex]
\;\; x_1 = x_\subinit, \quad G\left(x_{N+1}\right) = 0\quad \gamma>0 \end{array}\right. \\ 
\end{array}  
\end{equation*}

The key difference between $P1$ and $P2$ is the introduction of a switching variable $\theta_i \in \{0, 1\}$. Each $\theta_i$ represents a control mode: $\theta_i = 1$ indicates that the control input $u_i$ is active and bounded within $[\rho_\submin, \rho_\submax]$, while $\theta_i = 0$ indicates that $u_i$ is inactive and set to zero. This formulation is particularly useful for representing on-off control systems and extends the feasible region of $P1$. The objective function uses $l_i\left(\max(\|u_i\|, \rho_\text{min})\right)$, which serves as an approximation of $l_i\left(\|u_i\|\right)$. Furthermore, $u_i \in \mathbb{R}^2$ for $P2$, whereas the dimension of $u_i$ is not prescribed in $P1$.

Both $P1$ and $P2$ are non-convex and challenging to solve numerically. We thus construct convex relaxations whose solutions satisfy the original nonconvex constraints at almost all time grid points.

For $P1$, we apply the standard LCvx for the pointing constraints \cite{acikmese2013}. The resulting convexified problem reads: \vspace{0.2cm}

\textbf{Problem 3 (P3)}:
\label{equ:convexified_original_LCvx}
\begin{equation*}
\begin{array}{ll}
\underset{x, u, \sigma}{\text{minimize}} & \quad \;m\left(x_{N+1}\right) + \sum_{i=1}^N l_i\left(\sigma_i\right) \\ 
\text{subject to}  & \left|\begin{array}{l}  \left.\begin{array}{l} x_{i+1} = A x_{i} + B u_i + z_i \\ 
 \rho_\submin \leq \sigma_i \leq \rho_\submax \\ 
\xi^{\top} u_i \geq \gamma \sigma_i \\ 
 \sigma_i \geq \|u_i\| \end{array} \right| \forall \, i \in \intset{1}{N} \\[4.5ex] 
\;\; x_1 = x_\subinit, \quad G\left(x_{N+1}\right) = 0 \end{array} \right.
\end{array}
\end{equation*}
Here, $\sigma_i \in \mathbb{R}$ for all $i$, and we define $\sigma = \{\sigma_1, \sigma_2, \cdots, \sigma_N\}$. All other variables and coefficients are defined as in $P1$.

$P2$ is an MI optimization problem that cannot be solved directly using continuous optimization techniques. To address this, we recast $P2$ as follows. \vspace{0.2cm}

\textbf{Problem 4 (P4)}:
\label{equ:convexified_original_LCvx2}
\begin{equation*}
\begin{array}{ll}
\underset{x, u, \sigma}{\text{minimize}} & \quad \; m\left(x_{N+1}\right) + \sum_{i=1}^N l_i\left(\sigma_i\right) \\ 
\text{subject to} & \left| \begin{array}{l} \left. \begin{array}{l} x_{i+1} = A x_{i} + B u_i + z_i \\ 
\rho_\submin \leq \sigma_i \leq \rho_\submax \\ 
\xi^{\top} u_i \geq \gamma \|u_i\| \\ 
\sigma_i \geq \|u_i\|, \ u_i \in \mathbb{R}^2 \end{array} \right| \forall \,i \in \intset{1}{N}\\[4.5ex] 
 \;\;x_1 = x_\subinit, \quad G\left(x_{N+1}\right) = 0\quad \gamma>0 \end{array}\right.
\end{array}
\end{equation*}

The feasible region for $u_i$, given a specific value of $\sigma_i$, is a sector-shaped region, as shown in Fig.~\ref{fig:contr_4}. \blue{ The directions \( OA \) and \( OB \) are the only two whose inner product with \( {\xi} \) equals \( \gamma \). Let \( {\xi}_1 \) and \( {\xi}_2 \) denote the unit vectors perpendicular to \( OA \) and \( OB \), respectively, and both have negative inner product with \( {\xi} \). The inner product between \( OA \) and \( OB \) is positive, and the direction from \( OA \) to \( OB \) is clockwise. The feasible region is the sector from \( OA \) to \( OB \) in the clockwise direction.}

\blue{Note that the objective of $P4$ involves $l_i(\sigma_i)$, which is lower bounded by $l_i(\rho_{\min})$. Consequently, the cost of $P4$  is insensitive to any control input below $\rho_{\min}$, which necessitates using $l_i\left(\max\left(\|u_i\|, \rho_{\min}\right)\right)$ in $P2$ instead of $l_i(\|u_i\|)$ to ensure the validity of LCvx.
}
\begin{figure}[tp]
\centering
    \begin{subfigure}[t]{0.40\columnwidth}
        \centering
        \begin{adjustbox}{width=\textwidth}
       \begin{tikzpicture}
        \draw[fill=ultramarine!80,opacity=0.50] (-1,0) arc (120:60:2cm and 2cm) -- cycle;
        \node(origin) at (0, -2.2) {\scriptsize O};
        \node(originb) at (-1.2, 0) {\scriptsize A};
        \node(origina) at (1.2, 0) {\scriptsize B};
        \draw[->] (origin) to (0, 0.8);
        \draw[-] (origina) to (originb);
        \draw[-] (-1,0) arc (120:60:2cm and 2cm);
        \draw[dashed] (origina.west) to (origin.north);
        \draw[dashed] (originb.east) to (origin.north);
        \node(lab1) at (0.2, 0.8) {\scriptsize $\xi$};
    \end{tikzpicture}
    \end{adjustbox}
    \caption{$P$3 -- 2D case}
    \label{fig:contr_3}
    \end{subfigure}
    \hspace{0.1\columnwidth}
    \begin{subfigure}[t]{0.40\columnwidth}
    \centering
        \begin{adjustbox}{width=\textwidth}
       \begin{tikzpicture}
        \draw[color=white!00,fill=ultramarine!80,opacity=0.50] (-1,0) arc (120:60:2cm and 2cm) -- (0, -2) -- cycle;
        \node(origin) at (0, -2.2) {\scriptsize O};
        \node(originb) at (-1.2, 0) {\scriptsize A};
        \node(origina) at (1.2, 0) {\scriptsize B};
        \draw[->] (origin) to (0, 0.8);
        \draw[dashed] (origina) to (originb);
        \draw[-] (-1,0) arc (120:60:2cm and 2cm);
        \node(ob) at ($ (-0.6,0)!0.5!(0,-2) $) {};
        \node(oa) at ($ (0.6,0)!0.5!(0,-2) $) {};
        \draw[->] (ob) to (-1.3,-1.55);
        \draw[->] (oa) to (1.3,-1.55);
        \draw[-] (origina.west) to (origin.north);
        \draw[-] (originb.east) to (origin.north);
        \node(lab1) at (0.2, 0.8) {\scriptsize $\xi$};
        \node(lab1) at (1.2, -1.20) {\scriptsize $\xi_2$};
        \node(lab1) at (-1.2, -1.2) {\scriptsize $\xi_1$};
    \end{tikzpicture}
        \end{adjustbox}
    \caption{$P$4}
    \label{fig:contr_4}
    \end{subfigure}
    \caption{Admissible control sets for a given $\sigma$ }
\end{figure}

$P3$ and $P4$ can be consistently reformulated as: \vspace{0.2cm}

\textbf{Problem 5 (P5)}:
\label{equ:convexified_simplified_LCvx}
\begin{equation*}
\begin{array}{ll}
\underset{x, u, \sigma}{\text{minimize}} &  \quad \; m\left(x_{N+1}\right) + \sum_{i=1}^N \left[l_i\left(\sigma_i\right) + I_V( u_i, \sigma_i)\right] \\ 
\text{subject to}& \left| \begin{array}{l} x_{i+1} = A x_{i} + B u_i + z_i, \quad \forall i \in \intset{1}{N} \\  
x_1 = x_\subinit, \quad G\left(x_{N+1}\right) = 0  \end{array}\right.
\end{array}
\end{equation*}
where the \textit{indicator function} $I_V(u, \sigma)$ of the convex set $V$ is 0 for $(u, \sigma) \in V$ and $+\infty$ otherwise \cite{borwein2006convex}. 
\if\toggleFScomm1
\blue{The expressions of $V$ for $P3$ and for $P4$ follow.}
\else
The sets $V$ for $P3$ are defined as:
\fi
\begin{equation*}
\begin{aligned}
V_{P3} &= \left\{ (\bar{u}, \bar{\sigma}) \in \mathbb{R}^{n_u} \times \mathbb{R} \,\middle|\,
\begin{array}{l}
\rho_{\min} \leq \bar{\sigma} \leq \rho_{\max}, \\
\xi^\top \bar{u} \geq \gamma \bar{\sigma},\quad
\bar{\sigma} \geq \|\bar{u}\|
\end{array}
\right\}, \\
V_{P4} &= \left\{ (\bar{u}, \bar{\sigma}) \in \mathbb{R}^2 \times \mathbb{R} \,\middle|\,
\begin{array}{l}
\rho_{\min} \leq \bar{\sigma} \leq \rho_{\max}, \\
\xi^\top \bar{u} \geq \gamma \|\bar{u}\|,\quad
\bar{\sigma} \geq \|\bar{u}\|
\end{array}
\right\}.
\end{aligned}
\end{equation*}

Furthermore, we set
\begin{equation*}
\begin{aligned}
V_{P3}(\bar{\sigma}) &= \left\{ \bar{u} \in \mathbb{R}^{n_u} \;\middle|\; (\bar{u}, \bar{\sigma}) \in V_{P3} \right\}, \\
V_{P4}(\bar{\sigma}) &= \left\{ \bar{u} \in \mathbb{R}^{n_u} \;\middle|\; (\bar{u}, \bar{\sigma}) \in V_{P4} \right\}.
\end{aligned}
\end{equation*}

After removing redundant boundary conditions, the affine equality constraints can be combined into a single affine constraint \( H(x, u) = 0 \), with Jacobian \( \nabla H \). We then make the following assumption:
\begin{assumption}
\label{assumption:jacobian H}
$\nabla H$ is full rank.
\end{assumption}

We define LCvx as \textit{valid} for \( P3 \) at the time grid point \( i \) if the optimal control \( u_i^* \) satisfies the control constraints in \( P1 \). \if\toggleFScomm1
\blue{Similarly, LCvx is \textit{valid} for \( P4 \) at the time grid point \( i \) if the optimal solution \( u_i^* \) satisfies the control constraints in \( P2 \) for either $\theta_i=1$ or $\theta_i=0$.}
\fi

\blue{Let \( (\bar{u}, \bar{\sigma}) \) denote the control at a generic time grid point~\( i \). We now provide a sufficient condition for the validity of LCvx for \( P3 \) and \( P4 \) at \( i \). The following lemma highlights a key difference between CT-LCvx and DT-LCvx: the "if" part always holds in CT-LCvx~\cite{acikmese2013}, but not in DT-LCvx.}


\begin{lemma}[LCvx Sufficient Condition]
\label{lemma:normalcone}
For a point \((\bar u, \bar \sigma) \in V_{P3}\), LCvx is valid if 
\[
N_{V_{P3}(\bar \sigma)}(\bar u) \not\subseteq \{\lambda \xi  \mid \lambda \in \mathbb{R}\}.
\]
Similarly, for a point \((\bar u, \bar \sigma) \in V_{P4}\), LCvx is valid if 
\[
N_{V_{P4}(\bar \sigma)}(\bar u) \not\subseteq \{\lambda \xi_1 \mid \lambda \in \mathbb{R}\} \cup \{\lambda \xi_2 \mid \lambda \in \mathbb{R}\}.
\]
\end{lemma}
\begin{proof}
We first prove the statement for $P3$. The feasible region for \(\bar{u}\), given $\bar \sigma$, is a high-dimensional spherical cap with base perpendicular to \(\xi\). Fig.~\ref{fig:contr_3} illustrates the 2D case. LCvx holds (a) if \(\bar{u}\) is not in the interior of \(V_{P3}(\bar{\sigma})\), whose normal cone is \(\{0\}\), and (b) if \(\bar{u}\) is not on the base of the cap (the segment AB in the 2D case), whose normal cone is a subset of \(\{\lambda \xi \mid \lambda \in \mathbb{R}\}\). Thus, LCvx holds if \(N_{V_{P3}(\bar{\sigma})}(\bar{u}) \nsubseteq \{\lambda \xi \mid \lambda \in \mathbb{R}\}\).  

Next we show the lemma for \(P4\). Let us refer to Fig.~\ref{fig:contr_4}. There are two scenarios in which the constraints in \(P4\) can be violated: (a) \(\bar u\) is in the interior of \(V_{P4}(\bar \sigma)\), or (b) \(\bar u\) lies on the line segment \(OA\) or \(OB\), but not at points \(O, A,\) or \(B\).

If (a) happens, then \(N_{V_{P4}(\bar \sigma)}(\bar u) = \{0\}\)~\cite[Corollary 6.44]{bauschke2011convex}. If (b) happens, then \(N_{V_{P4}(\bar \sigma)}(\bar u)\) is \(\{\lambda \xi_1 \mid \lambda \in \mathbb{R}_+\}\) or \(\{\lambda \xi_2 \mid \lambda \in \mathbb{R}_+\}\).  
Thus, LCvx holds if 
\[
N_{V_{P4}(\bar \sigma)}(\bar u)\nsubseteq \{\lambda \xi_1 \mid \lambda \in \mathbb{R}\} \cup \{\lambda \xi_2 \mid \lambda \in \mathbb{R}\}.
\] 
\end{proof}

\section{Theoretical guarantee for LCvx Validity}
\label{sec:prooflcvx}

In this section, we formally establish the validity of LCvx for the solutions of $P3$ and $P4$ and derive an upper bound on the number of temporal grid points where the optimal solutions to $P3$ and $P4$ may violate the non-convex constraints of $P1$ and $P2$ respectively.

We begin by assuming that Slater's condition holds for $P5$, which also holds for both $P3$ and $P4$.

\begin{assumption}[Slater's Condition]
\label{assumption: slater}
For $P5$, there exists a combination of $(x, u, \sigma)$ such that $( u_i, \sigma_i)$ lies in the interior of the set $V$ for all $i$ and $(x_i, u_i)$ satisfies all affine constraints.
\end{assumption}

Slater's condition ensures the validity of the following theorem for the Lagrangian defined as $L(x, u, \sigma; \eta, \mu_1, \mu_2) := $
\begin{equation*}
\begin{aligned}
 & \, m\left(x_{N+1}\right) + \sum_{i=1}^{N} \left[l_i\left(\sigma_i\right) + I_{V}\left(u_i, \sigma_i\right)\right] + \mu_1^\top\left(x_1 - x_{\text{init}}\right)\\
 &+ \mu_2^\top G\left(x_{N+1}\right)+ \sum_{i=1}^{N} \eta_i^\top(-x_{i+1} + A x_i + B u_i + z_i) \\
\end{aligned}
\end{equation*}
Here, the dual variables are defined as $\eta_i \in \mathbb{R}^{n_x}$ for $i = 1, 2, \cdots, N$, and they are concatenated to form the vector $\eta = [\eta_1^\top, \eta_2^\top, \cdots, \eta_N^\top]^\top \in \mathbb{R}^{n_x N}$. Furthermore, $\mu_1 \in \mathbb{R}^{n_x}$ and $\mu_2 \in \mathbb{R}^{n_G}$ are the dual variables associated with the initial state and the boundary constraints, respectively.


\begin{theorem}[Theorem~9.4, 9.8, and 9.14 in \cite{clarke2013functional}]
\label{thm: lagrangian}
Given a solution $(x^*, u^*, \sigma^*)$ of P5 and Assumption \ref{assumption: slater}, there exist dual variables $(\eta^*, \mu_1^*, \mu_2^*)$ such that
\begin{equation}
\label{equ: min lagrangian}
\left(x^*, u^*, \sigma^* \right)= \operatorname{arg} \min_{x, u, \sigma} L\left(x, u, \sigma; \eta^*, \mu_1^*, \mu_2^*\right).
\end{equation}
\end{theorem}

From Eq. \eqref{equ: min lagrangian}, we derive:
\begin{align}
  &  \sigma^* = \operatorname{arg} \min_\sigma L\left(x^*, u^*, \sigma; \eta^*, \mu_1^*, \mu_2^*\right),  \\
  & u^*= \operatorname{arg} \min_u L\left(x^*,u, \sigma^*; \eta^*, \mu_1^*, \mu_2^*\right), \label{minimization on u} \\
  & x^* = \operatorname{arg} \min_x L\left(x, u^*, \sigma^*; \eta^*, \mu_1^*, \mu_2^*\right). \label{minimization on x}
\end{align}
Control $u_i$ only appears in terms $I_V(u_i, \sigma_i)$ and $\eta_i^\top B u_i$ in the Lagrangian $L$; we have:

\begin{equation}
\setlength{\arraycolsep}{1.5pt}
\begin{array}{rl}
u_i^* & \displaystyle =\arg \min _{u_i} I_V\left(u_i, \sigma_i^*\right)+\eta_i^{* \top} B u_i \\[1.5ex]
& \displaystyle =\arg \min _{u_i} I_{V\left(\sigma_i^*\right)}\left(u_i\right)+\eta_i^{* \top} B u_i
\end{array}
\label{equ:ui}
\end{equation}
 where the set $V(\sigma_i^*)$ is defined as 
$$
V(\sigma_i^*)= \{\Bar{u} \in \mathbb{R}^{n_u} \;\;| \;\; (\Bar{u}, \sigma_i^*) \in V \}.
$$
\begin{theorem}[Theorem 10 of \cite{Luo2024-cy}]
\label{theorem: LCvx lambda B}
Under Assumption \ref{assumption: slater} for $P5$, we have 
$\eta_i^* = A^{\top (N-i)} \eta_N$ and $-B^{\top}\eta_i ^{*} \in N_{V(\sigma_i^*)}\left(u_i^*\right).$
\end{theorem}
\begin{proof}
 Given that the right-hand side function of Eq. \eqref{equ:ui} is convex in \( u_i \) and \( u_i^* \) is a minimizer of this convex function, it follows that $
-B^{\top}\eta_i ^{*} \in N_{V(\sigma_i^*)}\left(u_i^*\right)
$ \cite[Proposition 2.1.2]{borwein2006convex}. 

For the evolution of \( \eta_i \), considering Eq. \eqref{minimization on x}, we take the partial derivative of \( L \) with respect to \( x_i \) for all \( i \), yielding:
\begin{subequations}
\begin{align}
x_{N+1}: & \quad \nabla m\left(x_{N+1}\right) + \mu_2^\top \nabla G\left(x_{N+1}\right) = \eta^*_N, \label{eq:xN}\\
x_i: & \quad \eta_i^{*\top} A - \eta_{i-1}^{*\top} = 0, \label{eq:xi}\\
x_1: & \quad \mu_1^\top + \eta_1^{*\top} A = 0. \label{eq:x0}
\end{align}
\end{subequations}
Iteratively applying Eq. \eqref{eq:xi} leads to  \( \eta^*_i = A^{\top(N-i)} \eta^*_N \). 
\end{proof}



Hence, we have the following sufficient condition for LCvx for a specific grid point $i$.

\begin{theorem}
\label{thm: LCvx sufficient}
LCvx is valid for $P3$ at time grid point $i$ if 
\[
-B^{\top} \eta_i^{*} \neq \lambda \xi , \quad \text{for all } \lambda \in \mathbb{R}.
\]
Similarly, LCvx is valid for $P4$ at time grid point $i$ if 
\[
-B^{\top} \eta_i^{*} \neq \lambda \xi_1 \; \wedge \; -B^{\top} \eta_i^{*} \neq \lambda \xi_2, \quad \text{for all } \lambda \in \mathbb{R}.
\]
\end{theorem}
\begin{proof}
    This theorem holds as a direct consequence of Lemma \ref{lemma:normalcone} and Theorem \ref{theorem: LCvx lambda B}. 
\end{proof}

%


We now analyze the validity of LCvx for the full trajectory of $P3$ and $P4$, beginning with the definition of controllability.

\begin{definition}[Controllability]
\label{definition:controllability}
The matrix pair \( (A, B) \), where \( A \in \mathbb{R}^{n_x \times n_x} \) and \( B \in \mathbb{R}^{n_x \times n_u} \), is said to be controllable if the controllability matrix 
\[
\mathcal{C} = \begin{bmatrix}
B & AB & A^2 B & \cdots & A^{n_x-1} B
\end{bmatrix}
\]
has full row rank, i.e. $\operatorname{rank}(\mathcal{C}) = n_x.$

\end{definition}

%
\blue{We use the following controllability assumption for $P3$ and $P4$, which appears in CT-LCvx \cite{acikmese2013} and is verifiable offline.}

\begin{assumption}[Controllability for $P3$ and $P4$]
\label{assumption:controllability1}
In $P3$, let $M \in \mathbb{R}^{n_u \times (n_u - 1)}$ be a matrix whose columns form a basis for the orthogonal complement of $\xi$; the pair $\{A, BM\}$ is controllable. In $P4$, let $w_1, w_2 \in \mathbb{R}^2$ be unit vectors orthogonal to $\xi_1$ and $\xi_2$, respectively; both pairs $\{A, Bw_1\}$ and $\{A, Bw_2\}$ are controllable.
\end{assumption}

\blue{The following theorem for \( P3 \) motivates Assumption~\ref{assumption:controllability1}.}

\begin{theorem}
\label{thm: last n}
Suppose Assumptions \ref{assumption: slater} and \ref{assumption:controllability1} hold. \blue{If $\eta_N\neq 0, $  the optimal solution to $P3$ does not violate LCvx at all of the last $n_x$ time grid points.}
\end{theorem}

\begin{proof}
Under Assumption \ref{assumption:controllability1} and that \blue{$\eta_N \neq 0$}, we have
\[
\eta_N^\top \begin{bmatrix}
BM & ABM & \cdots & A^{n_x-1} BM
\end{bmatrix} \neq 0,
\] which implies that there exists a nonnegative integer $i < n_x$ such that $
\eta_N^\top A^i BM \neq 0.$
Since the column space of $M$ is orthogonal to $\xi$, it follows that $\eta_N^\top A^i B$ is neither parallel to $\xi$ nor zero. Furthermore, by Theorem \ref{theorem: LCvx lambda B}, we have $\eta_N^\top A^i B = \eta_{N-i}^\top B.$ Therefore, Theorem \ref{thm: LCvx sufficient} ensures that LCvx is valid at the $(N - i)$th time grid point. \blue{Hence, LCvx holds at least at one of the last $n_x$ time grid points.}
\end{proof}

\blue{Since $\eta_N$ is only known after solving the convex problem, we impose a sufficient condition to ensure $\eta_N \neq 0$ in Theorem~\ref{thm: last n}.}

\begin{assumption}[Transversality]
\label{assumption: Transversality}
\blue{For the optimal solution to $P5$ (and thus to $P3$ and $P4$), we assume that either $\sigma_i^* \neq \rho_{\min}$ for some $i$, or the terminal state $x_{N+1}$ does not minimize $m(x)$ subject to $G(x) = 0$.}
\end{assumption}

\begin{lemma}
\label{lemma: transversality}
\blue{Assumption~\ref{assumption: Transversality} implies $\eta_N \neq 0$.}
\end{lemma}

\begin{proof}
\blue{We argue by contradiction. Suppose $\eta_N^* = 0$. Then Theorem~\ref{theorem: LCvx lambda B} implies $\eta_i^* = 0$ for all $i$, and $\nabla m(x_{N+1}^*) + \mu_2^\top \nabla G(x_{N+1}^*) = 0$, so $x_{N+1}^*$ minimizes $m(x)$ subject to $G(x) = 0$. By Theorem~\ref{thm: lagrangian}, $(\sigma^*, u^*)$ minimizes $\sum_{i=1}^N l_i(\sigma_i) + \sum_{i=1}^N I_V(u_i, \sigma_i)$, which, by the monotonicity of $l$, yields $\sigma_i^* = \rho_{\min}$ for all $i$. This contradicts Assumption~\ref{assumption: Transversality}.}
\end{proof}

\blue{Assumption~\ref{assumption: Transversality} is violated if there exists a control below the minimum threshold that achieves the minimum final-state cost. Such violations, checkable post-computation, can be resolved via a bisection strategy; see~\cite[Section 5]{Luo2024-cy} for a full discussion.}


The remainder of this chapter extends Theorem \ref{thm: last n} to cover the entire trajectory of $P3$ and derives the corresponding result for $P4$. Importantly, controllability alone is insufficient for this generalization. Counterexamples for problems without pointing constraints exist \cite[Section 6]{Luo2024-cy}. \blue{To overcome this limitation, we introduce a stronger controllability condition that requires controllability at multiple, arbitrarily chosen time steps along the trajectory.}

\begin{assumption}
\label{assumption: full_controllablity}
\blue{Let $\{P_1, P_2, \ldots, P_{n_x}\} \subset \intset 1 N$ be any set of $n_x$ distinct integers. We require that the matrix
\[
\begin{bmatrix}
A^{N-P_1} \bar{B} & A^{N-P_2} \bar{B} & \cdots & A^{N-P_{n_x}} \bar{B}
\end{bmatrix}
\]
has full row rank. Here, $\bar{B} = BM$ as defined in Assumption~\ref{assumption:controllability1} for $P3$, and $\bar{B} \in \{Bw_1, Bw_2\}$ as defined in Assumption~\ref{assumption:controllability1} for $P4$.
}\end{assumption}
\begin{remark}
 \label{rem:full_controllability}
\blue{Assumption~\ref{assumption:controllability1} implies Assumption~\ref{assumption: full_controllablity} up to an arbitrarily small random perturbation to $A$. In practice, such perturbation is typically unnecessary, and directly verifying Assumption~\ref{assumption:controllability1} is sufficient; see~\cite[Theorem 20]{Luo2024-cy} for a rigorous statement.}   
\end{remark}

\begin{theorem}[Main Theorem]
\label{thm: main}
\blue{Under Assumptions~\ref{assumption:jacobian H}, \ref{assumption: slater}, \ref{assumption: Transversality}, and~\ref{assumption: full_controllablity}, the optimal solution to $P3$ violates the nonconvex constraints in $P1$ at most \(n_x - 1\) times, and the optimal solution to $P4$ violates the nonconvex constraints in $P2$ at most \(2n_x - 2\) times.}
\end{theorem}

\begin{proof}
\blue{For $P3$, Assumption~\ref{assumption: full_controllablity} and Lemma \ref{lemma: transversality} allow for the same reasoning as in Theorem~\ref{thm: last n}: LCvx holds at least once in any set of $n_x$ distinct time steps. Hence, if there were $n_x$ or more violations, one could select $n_x$ violating points—contradicting the previous conclusion.}

\blue{For $P4$, Theorem~\ref{thm: LCvx sufficient} shows that LCvx is violated at time step $i$ only if $B^\top \eta_i^*$ is parallel to $\xi_1$ or $\xi_2$, i.e., $w_1^\top B^\top \eta_i^* = 0$ or $w_2^\top B^\top \eta_i^* = 0$, where $w_1$ and $w_2$ are defined in Assumption~\ref{assumption:controllability1}. By Assumption~\ref{assumption: full_controllablity}, Lemma \ref{lemma: transversality} and Theorem~\ref{thm: last n}, any set of $n_x$ distinct time steps contains at least one $i$ such that $B^\top \eta_i^*$ is not parallel to $\xi_1$, and one such that it is not parallel to $\xi_2$. Hence, if LCvx were violated at $n_x$ or more time steps due to either condition, one could find $n_x$ violations of the same type—contradicting the previous conclusion. The total number of violations is thus at most $2n_x - 2$.}
\end{proof}
\begin{remark}
\blue{While Theorem~\ref{thm: main} bounds LCvx violations, they may still arise. A simple fix is to project the violating controls onto the feasible set, which incurs a final-state error of $O(t_f/N)$~\cite[Theorem 23]{Luo2024-cy}. In practice, such violations are rare, and moderate discretization typically suffices. One may either choose $N$ based on theoretical bounds or increase it adaptively if needed.}
 \end{remark}

\begin{remark}
\label{rem:heur}
\blue{
Theorem~\ref{thm: main} for $P4$ holds only in $\mathbb{R}^2$; in higher dimensions, ensuring that $-B^{\top}\eta_i^*$ avoids all cone normal vectors may require an infinite number of controllablity conditions; the resulting violation bound is thus infinity. To restore theoretical guarantees in higher dimensions, one may use polyhedral cones, which yield finitely many normal directions.}

\blue{
However, for control in $\mathbb{R}^3$, although lacking guarantees, our formulation performs well as a heuristic for solving $P2$, with violations observed at only a few time steps. To further promote engine shutdown, we suggest adding a term proportional to $l_i(\|u_i\|)$ to the objective of $P4$ yielding higher-quality heuristics; see the numerical section for a comparison of these $P4$ variants.
}
\end{remark}

\section{Numerical analyses}
\blue{We provide a numerical example for dual-mode OCPs. LCvx for pointing constraints has already been validated through both simulations~\cite{acikmese2013} and experiments \cite{acikmese2012g}: no validation of $P3$ is hence presented here. All results are reproducible at {\fontsize{9.3}{12}\selectfont\texttt{\url{https://github.com/UW-ACL/DT-LCvx-Pointing}}}.}
\label{sec:numerical}
\blue{\subsection{Theoretical verifications for control in $\mathbb{R}^2$}
Consider the flight of a quadrotor, landing at $\bar{r}_f \in \mathbb{R}^2$ with velocity $\bar{v}_f \in \mathbb{R}^2$; initial position, speed and time of flight $\bar{r}_0 \in \mathbb{R}^2$, $\bar{v}_0 \in \mathbb{R}^2$,  $\bar{t}_f \in \mathbb{R}$ are assigned. The vehicle is subject to gravitational acceleration $g \in \mathbb{R}^2$ and linear drag, with constant coefficient $k_d \in \mathbb{R}_+$. Acceleration $u \in \mathbb{R}^2$ constitutes the only control variable.
State, control matrices and external disturbances are time-invariant and are reported as follows.
\begin{align}
\label{eq:Amat} A_c = \left[ \begin{array}{cc} \bm{0}_{2\times2} & \bm{I}_{2\times2} \\ \bm{0}_{2\times2} & -k_d\bm{I}_{2\times2}\end{array}\right],& \quad A = e^{(t_f/N)A_c}\\ \label{eq:Bmat}   B_c = \left[ \begin{array}{cc} \bm{0}_{2\times2} & \bm{0}_{2\times1} \\ \bm{I}_{2\times2} & \bm{0}_{2\times1}\end{array}\right], & \quad B = \int_{0}^{t_f/N} e^{tA_c}\text{d}t \;B_c \\ 
\label{eq:wmat} z_c = \left[ \begin{array}{c} \bm{0}_{2\times1} \\ g\end{array}\right],& \quad z_i = \int_{0}^{t_f/N} e^{tA_c}\text{d}t \;z_c 
\end{align}
Finite bounds on the acceleration magnitude, $\rho_{\min}$ and $\rho_{\max}$, provide necessary margins for the attitude controller. While modeling the feedback control for attitude is beyond the scope of this paper, it is important to note that, in practice, the thrust required by the attitude controller can become unrealistically negative for certain motors. Enforcing a lower bound $\rho_{\min} > 0$ helps mitigate this risk. In addition, the tilt angle relative to a reference direction $\xi \in \mathbb{R}^2$ is constrained by a maximum allowable value $\varphi_{\max}$~\cite{Szmuk2017-ve}.
Data are summarized in Tab~\ref{tab:data}.}

\begin{table}[htbp]
    \centering
    \caption{Numerical example data}
    \label{tab:data}
    \begin{threeparttable}
    \begin{tabular}{cc}
    \toprule \toprule
    Physical quantity &  Value$^\dagger$ \\[0.5ex]
    \midrule
    $\bar{r}_0, \bar{v}_0$    & $[0.5, 1.0], [-1.6, 0.6]$\\[0.5ex]
    $\bar{r}_f, \bar{v}_f$    & $[0, 0], [0, 0]$\\[0.5ex]
    $\bar{t}_f$     & $4.7$ \\[0.5ex]
    $g$             & $[0,-1]$ \\[0.5ex]
    $k_d, \rho_\submin, \rho_\submax$           & $0.05, 1.2, 1.6$ \\[0.5ex]
    $\varphi_\submax, \xi$    & $55^\circ, [0,1]$\\[0.5ex]
    \bottomrule \bottomrule     
    \end{tabular}
    \begin{tablenotes}
    \item $^\dagger$ Normalized according to unitary initial height and gravity acceleration.
    \end{tablenotes}
    \end{threeparttable}
\end{table}
\begin{figure*}[h!]
    \begin{subfigure}[b]{0.49\textwidth} \includegraphics[width=\textwidth]{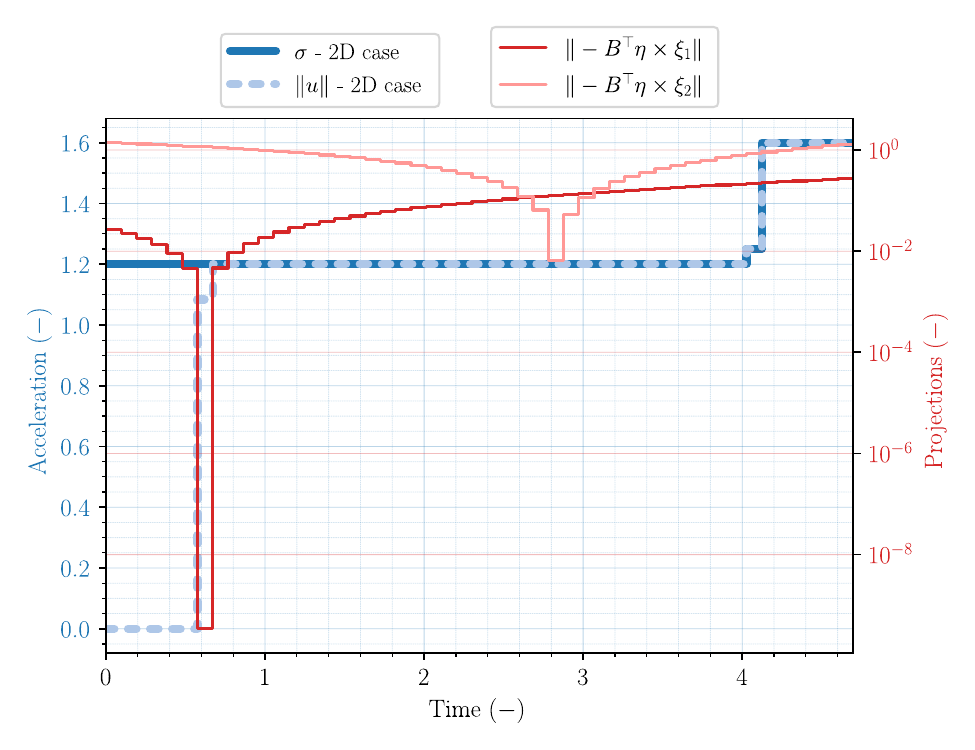}
        \caption{\blue{2D case - Control profiles and metrics corresponding to sufficient condition in Theorem \ref{thm: LCvx sufficient}.}}
        \label{fig:sufficiency}
    \end{subfigure}
    \hfill
        \begin{subfigure}[b]{0.49\textwidth}       \includegraphics[width=\textwidth]{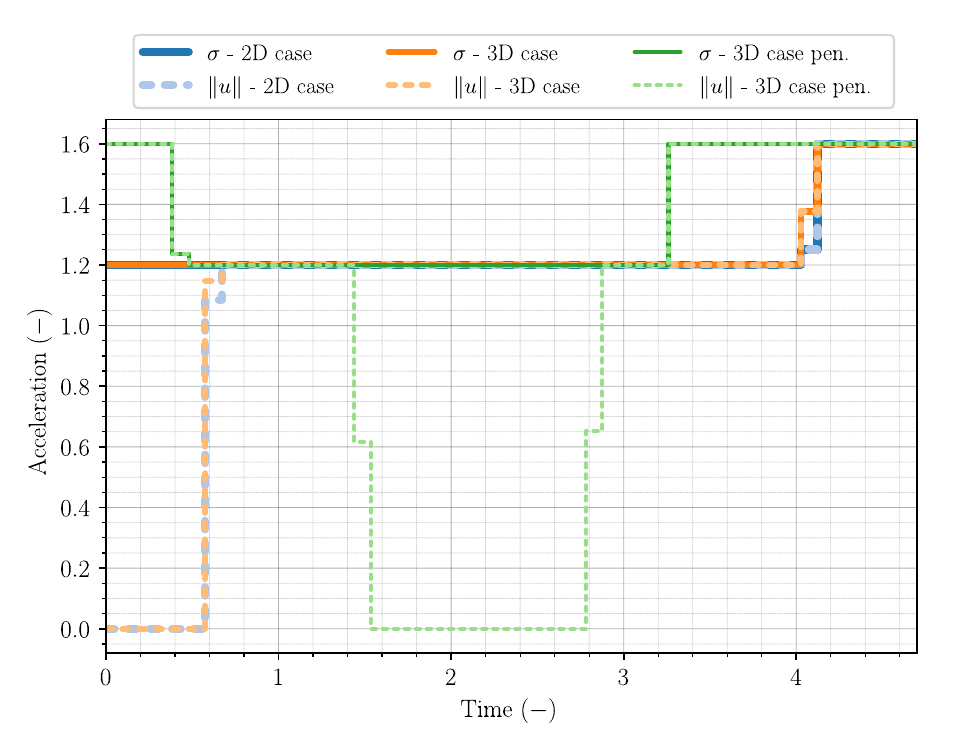}
        \caption{\blue{2D case versus heuristics from Remark \ref{rem:heur}. The additional cost term for the green profiles is $10\,l_i(\|u_i\|)$. }}
        \label{fig:heuristics}
    \end{subfigure}
    \caption{\blue{Results for the numerical experiment with pointing constraint from P4: on the left, the optimal controls and metrics measuring the  sufficient condition from Theorem \ref{thm: LCvx sufficient}. On the right a comparison of the 2D case ("$\bullet$ - 2D case") against the 2 heuristics  presented in Remark \ref{rem:heur}: orange lines ("$\bullet$ - 3D case") for the 3D case, green lines ("$\bullet$ - 3D case pen.") for the 3D case with penalization. QOCO used as solver \cite{Chari2025-ql}.}
    }
\end{figure*}
\blue{$N+1$ is fixed to 50.
We seek to reach the final state in the given time of flight and with control subject to mentioned constraints, while minimizing the control magnitude integral. Therefore, with respect to notation outlined in Sec. \ref{sec:2}, $l_i(\sigma_i) \coloneqq |\sigma_i|$. The problem is solved using pointing constraint expressions from both P3 and P4. With respect to the theoretical discussion, Assumptions \ref{assumption:jacobian H}, \ref{assumption: slater} are verified a priori; Assumption \ref{assumption: Transversality} is verified a posteriori. Assumption~\ref{assumption:controllability1} is satisfied by adding \(10^{-5}\) to the strict upper triangular part of \(A\), which has negligible physical impact under bounded control. Assumption~\ref{assumption: full_controllablity} is verified in Remark~\ref{rem:full_controllability}. Results in Fig.~\ref{fig:sufficiency} illustrate the connection between Theorem \ref{thm: LCvx sufficient} in Sec.~\ref{sec:prooflcvx} and the example: when LCvx holds, the dual variables satisfy $-B^\top \eta_i \neq \lambda \xi_1$ and $-B^\top \eta_i \neq \lambda \xi_2$, or equivalently, the cross products $\|-B^\top \eta_i \times \xi_1\| \neq 0$ and $\|-B^\top \eta_i \times \xi_2\| \neq 0$. In contrast, at the point where LCvx fails (around $t \approx 0.6$), the control lies on side $OA$, and we observe that $\|-B^\top \eta_i \times \xi_1\| = 0$, consistent with Theorem \ref{thm: LCvx sufficient}.  In this case, the pointing constraint from \(P3\) renders the problem infeasible.
 }
\blue{\subsection{Enhancements by heuristics for control in $\mathbb{R}^3$}
Although LCvx lacks theoretical guarantees, we present a three-dimensional control example to demonstrate that LCvx provides good heuristics for solving \( P2 \).
For this purpose, a second problem is designed by adding a third dimension; matrices in Eqs. \eqref{eq:Amat}, \eqref{eq:Bmat}, \eqref{eq:wmat} are augmented with an additional dimension; same data as in Tab. \ref{tab:data} are used; a position lateral offset of $\bar{r}_{0,z} = 0.45$ and speed offset of $\bar{v}_{0,z} = 0$ are added. The employed pointing constraint is a smooth constraint, as given in  \cite[Eq. 26]{Shaffer2024-yx}. A third problem is further solved with the 3D extension and a modified objective function: as noted in Remark \ref{rem:heur}, an additional term proportional to $l_i(\|u_i\|)$ in the objective of $P4$ encourages engine shutdown to save fuel. For our example, we have chosen a weight for $l_i(\|u_i\|)$ equal to 10. A comparison of the controls is reported in Fig. \ref{fig:heuristics}. A comparison between the profile of the 3D case and that of the 2D case shows that only one violation occurs in the 3D case, and it takes place at the same time grid point. As a general heuristic, the transition from a 2D to a 3D setting does not significantly affect the number of violations. A comparison between the 3D case and its penalized version supports Remark~\ref{rem:heur}: the added term promotes engine shutdowns. Notably, only two violations by LCvx are observed in the penalized case, around times 1.4 and 2.8. Solution trajectories are reported in Fig. \ref{fig:trajectories}. The particular structure of the control profile for the third problem determines a different flown trajectory with respect to the first and second problems.}
\begin{figure}[htbp]
    \centering
    \includegraphics[width=0.45\textwidth]{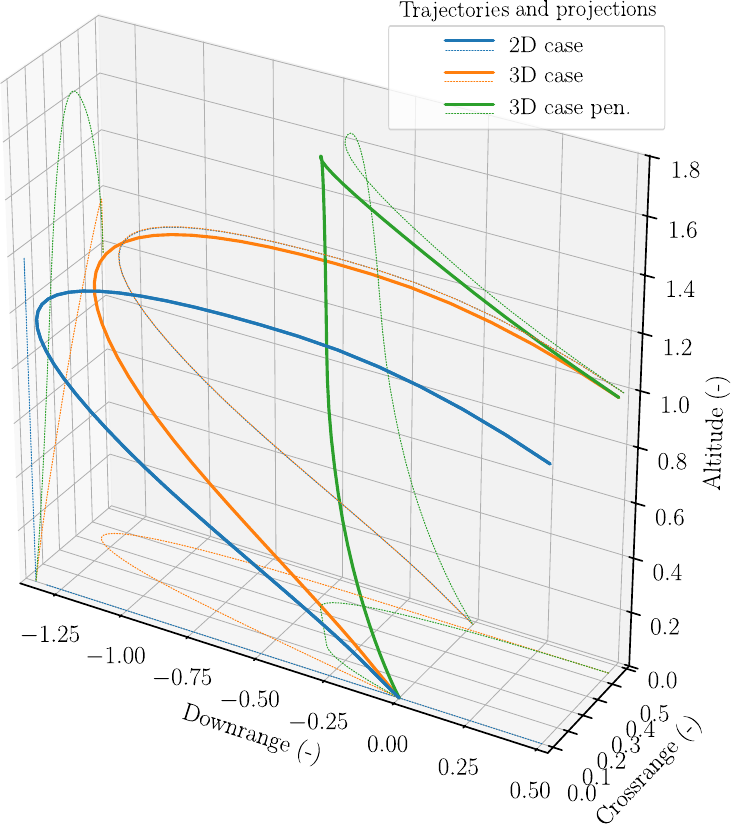}
    \caption{Comparison of the optimal trajectories for the numerical experiment with pointing constraint from P4. The 2D solution ("{\boldmath$\bullet$} - 2D case") is plotted against the two heuristics in Remark~\ref{rem:heur}: orange lines ("{\boldmath$\bullet$} - 3D case") for the 3D case, and green lines ("{\boldmath$\bullet$} - 3D case pen.") for the penalized 3D case.}
    \label{fig:trajectories}
\end{figure}

\section{Conclusion}
\label{sec:concl}
This paper extended DT-LCvx methods for optimal control with pointing constraints, including a novel mixed-integer formulation in which control is off or restricted to an annular sector. Under Slater's condition, controllability, and transversality, we proved that constraint violations occur at no more than $n_x - 1$ and $2n_x - 2$ grid points for the classical and new formulations in $\mathbb{R}^2$, respectively. Numerical results on a landing scenario demonstrate that the new formulation also provides a valid heuristic when control lies in $\mathbb{R}^3$.

\printbibliography

@book{borwein2006convex,
  title={Convex Analysis and Nonlinear Optimization},
  author={Borwein, Jonathan and Lewis, Adrian},
  year={2006},
  publisher={Springer}
}

@book{clarke2013functional,
  title={Functional Analysis, Calculus of Variations and Optimal Control},
  author={Clarke, Francis},
  volume={264},
  year={2013},
  publisher={Springer}
}

@book{Nesterov1994,
	author = "{Nesterov, Y. and Nemirovskii, A.}", 
title = "{Interior-Point Polynomial Algorithms in
Convex Programming}", 
publisher = "{Society for Industrial and Applied Mathematics}", 
year = "1994",
}

@article{malyuta2022convex,
  title={Convex optimization for trajectory generation: A tutorial on generating dynamically feasible trajectories reliably and efficiently},
  author={Malyuta, Danylo and Reynolds, Taylor P and Szmuk, Michael and Lew, Thomas and Bonalli, Riccardo and Pavone, Marco and A{\c{c}}{\i}kme{\c{s}}e, Beh{\c{c}}et},
  journal={IEEE Control Systems Magazine},
  volume={42},
  number={5},
  pages={40--113},
  year={2022},
  publisher={IEEE}
}

@article{acikmese2007convex,
  title="{Convex programming approach to powered descent guidance for Mars landing}",
  author={{\acikmese}, Behcet and Ploen, Scott},
  journal={Journal of Guidance, Control, and Dynamics},
  volume={30},
  number={5},
  pages={1353--1366},
  year={2007}
}

@article{blackmore2010,
  title     = "{Minimum-landing-error powered-descent guidance for Mars landing
               using convex optimization}",
  author    = "Blackmore, Lars and A\c{c}\i kme\c{s}e, Beh\c{c}et and Scharf,
               Daniel P",
  journal   = "Journal of Guidance, Control, and Dynamics",
  publisher = "American Institute of Aeronautics and Astronautics (AIAA)",
  volume    =  33,
  number    =  4,
  pages     = "1161--1171",
  month     =  jul,
  year      =  2010,
  doi       = "10.2514/1.47202",
}

@article{acikmese2013,
  title     = "{Lossless convexification of nonconvex control bound and pointing
               constraints of the soft landing optimal control problem}",
  author    = "A\c{c}\i kme\c{s}e, Beh\c{c}et and Carson, John M and Blackmore, Lars",
  journal   = "IEEE Transactions on Control Systems Technology",
  publisher = "Institute of Electrical and Electronics Engineers (IEEE)",
  volume    =  21,
  number    =  6,
  pages     = "2104--2113",
  month     =  nov,
  year      =  2013,
  doi       = "10.1109/tcst.2012.2237346"
}

@article{harris2014lossless,
  title={Lossless convexification of non-convex optimal control problems for state constrained linear systems},
  author={Harris, Matthew W and A{\c{c}}{\i}kme{\c{s}}e, Beh{\c{c}}et},
  journal={Automatica},
  volume={50},
  number={9},
  pages={2304--2311},
  year={2014},
  publisher={Elsevier}
}

@inproceedings{harris2013lossless,
  title={Lossless convexification for a class of optimal control problems with linear state constraints},
  author={Harris, Matthew W and A{\c{c}}{\i}kme{\c{s}}e, Beh{\c{c}}et},
  booktitle={52nd IEEE Conference on Decision and Control},
  pages={7113--7118},
  year={2013},
  organization={IEEE}
}

@article{kunhippurayil2021lossless,
  title={Lossless convexification of optimal control problems with annular control constraints},
  author={Kunhippurayil, Sheril and Harris, Matthew W and Jansson, Olli},
  journal={Automatica},
  volume={133},
  pages={109848},
  year={2021},
  publisher={Elsevier}
}

@ARTICLE{malyuta2020,
  title     = "{Lossless convexification of optimal control problems with
               semi-continuous inputs}",
  author    = "Malyuta, Danylo and A\c{c}\i kme\c{s}e, Beh\c{c}et",
  journal   = "IFAC-PapersOnLine",
  publisher = "Elsevier BV",
  volume    =  53,
  number    =  2,
  pages     = "6843--6850",
  year      =  2020,
  doi       = "10.1016/j.ifacol.2020.12.341",
}

@ARTICLE{Harris2021,  
author={Harris, Matthew W.},  
journal={IEEE Transactions on Automatic Control},   
title={Optimal Control on Disconnected Sets Using Extreme Point Relaxations and Normality Approximations},   
year={2021},  
volume={66},  
number={12}, 
pages={6063--6070},  
doi={10.1109/TAC.2021.3059682}
}

@ARTICLE{Kazu2023,
  author={Echigo, Kazuya and Hayner, Christopher R. and Mittal, Avi and Sarsilmaz, Selahattin Burak and Harris, Matthew W. and A\c{c}{\i}kme\c{s}e, Beh{\c{c}}et},
  journal={IEEE Control Systems Letters}, 
  title={Linear Programming Approach to Relative-Orbit Control With Element-Wise Quantized Control}, 
  year={2023},
  volume={7},
  number={},
  pages={3042--3047},
  doi={10.1109/LCSYS.2023.3289472}}

@article{yang2024,
  title     = "{Convex hull relaxation of optimal control problems with general
               nonconvex control constraints}",
  author    = "Yang, Runqiu and Liu, Xinfu",
  journal   = "IEEE Transactions on Automatic Control",
  publisher = "Institute of Electrical and Electronics Engineers (IEEE)",
  volume    =  69,
  number    =  6,
  pages     = "4028--4034",
  month     =  jun,
  year      =  2024,
  doi       = "10.1109/tac.2023.3342061"
}

@article{acikmese2012g,
  title="{{G-FOLD}: A real-time implementable fuel optimal large divert guidance algorithm for planetary pinpoint landing}",
  author={A\c{c}\i{}kme\c{s}e, Behcet and Casoliva, Jordi and Carson, John M and Blackmore, Lars},
  journal={Concepts and Approaches for Mars Exploration},
  volume={1679},
  pages={4193},
  year={2012}
}

@inproceedings{acikmese2013flight,
  title={Flight testing of trajectories computed by {G-FOLD}: Fuel optimal large divert guidance algorithm for planetary landing},
  author={A\c{c}\i{}kme\c{s}e, Behcet and Aung, M and Casoliva, Jordi and Mohan, Swati and Johnson, Andrew and Scharf, Daniel and Masten, David and Scotkin, Joel and Wolf, Aron and Regehr, Martin W},
  booktitle={AAS/AIAA Spaceflight Mechanics Meeting},
  pages={386},
  year={2013}
}

@misc{bauschke2011convex,
  title="{Convex Analysis and Monotone Operator Theory in Hilbert Spaces}",
  author={Bauschke, HH},
  year={2011},
  publisher={Springer-Verlag}
}

@ARTICLE{Luo2024-cy,
  title        = "{Revisiting lossless Convexification: Theoretical guarantees
                  for discrete-time optimal control problems}",
  author       = {Luo, Dayou and Echigo, Kazuya and A\c{c}\i{}kme\c{s}e,
                  Beh\c{c}et},
  journal      = "arXiv [math.OC]",
  year         =  2024,
  primaryClass = "math.OC",
  doi          = "10.48550/ARXIV.2410.09748"
}

@INPROCEEDINGS{Szmuk2017-ve,
  title     = "{Convexification and real-time on-board optimization for agile
               quad-rotor maneuvering and obstacle avoidance}",
  author    = {Szmuk, Michael and Pascucci, Carlo Alberto and Dueri, Daniel and
               A\c{c}\i kme\c{s}e, Beh\c{c}et},
  booktitle = "{2017 IEEE/RSJ IROS}",
  publisher = "IEEE",
  month     =  sep,
  year      =  2017,
  doi       = "10.1109/iros.2017.8206363"
}

@ARTICLE{Lu2013-fb,
  title     = "{Autonomous trajectory planning for rendezvous and proximity
               operations by conic optimization}",
  author    = "Lu, Ping and Liu, Xinfu",
  journal   = {Journal of Guidance, Control, and Dynamics},
  publisher = "American Institute of Aeronautics and Astronautics (AIAA)",
  volume    =  36,
  number    =  2,
  pages     = "375--389",
  month     =  mar,
  year      =  2013,
  doi       = "10.2514/1.58436"
}

@ARTICLE{Yang2024-rp,
  title     = "{Exact relaxation of nonconvex optimal control problems based on
               problem reconstruction}",
  author    = "Yang, Runqiu and Liu, Xinfu and Lin, Defu",
  journal   = {Journal of Guidance, Control, and Dynamics},
  publisher = "American Institute of Aeronautics and Astronautics (AIAA)",
  volume    =  47,
  number    =  4,
  pages     = "761--769",
  month     =  apr,
  year      =  2024,
  doi       = "10.2514/1.g007788"
}

@ARTICLE{Woodford2022-cr,
  title     = "{Geometric Properties of {Time-Optimal} Controls With State
               Constraints Using Strong Observability}",
  author    = "Woodford, Nathaniel and Harris, Matthew",
  journal   = "IEEE transactions on automatic control",
  publisher = "Institute of Electrical and Electronics Engineers (IEEE)",
  volume    =  67,
  number    =  12,
  pages     = "6881--6887",
  year      =  2022,
  doi       = "10.1109/tac.2021.3134627"
}

@INCOLLECTION{Shaffer2024-yx,
  title     = "{Implementation and Testing of Convex Optimization-based Guidance
               for Hazard Detection and Avoidance on a Lunar Lander}",
  author    = "Shaffer, Joshua and Owens, Chris and Klein, Theresa and Horchler,
               Andrew and Buckner, Samuel and Johnson, Breanna and Carson, John
               and Acikmese, Behcet",
  booktitle = "{{AIAA} {SCITECH} 2024 Forum}",
  month     =  apr,
  year      =  2024,
  doi       = "10.2514/6.2024-1584"
}

@ARTICLE{Malyuta2023-bh,
  title     = "{Fast homotopy for spacecraft rendezvous trajectory optimization
               with discrete logic}",
  author    = "Malyuta, Danylo and A\c{c}\i{}kme\c{s}e, Beh\c{c}et",
  journal   = {Journal of Guidance, Control, and Dynamics},
  publisher = "American Institute of Aeronautics and Astronautics (AIAA)",
  volume    =  46,
  number    =  7,
  pages     = "1262--1279",
  month     =  jul,
  year      =  2023,
  doi       = "10.2514/1.g006295"
}

@INPROCEEDINGS{Yoshimura2012-to,
  title     = "{Satellite position and attitude control by on-off thrusters
               considering mass change}",
  author    = "Yoshimura, Yasuhiro and Matsuno, Takashi and Hokamoto, Shinji",
  booktitle = "{AAS/AIAA Astrodynamics Specialist Conference}",
  pages     = "2953--2964",
  year      =  2012,}

@ARTICLE{Chari2025-ql,
  title        = "{QOCO: A quadratic objective conic optimizer with custom
                  solver generation}",
  author       = "Chari, Govind M and A\c{c}ikme\c{s}e, Beh\c{c}et",
  journal      = "arXiv [math.OC]",
  month        =  "16~" # mar,
  year         =  2025,
  primaryClass = "math.OC"
}

@article{liu2017survey,
  title={Survey of convex optimization for aerospace applications},
  author={Liu, Xinfu and Lu, Ping and Pan, Binfeng},
  journal={Astrodynamics},
  volume={1},
  pages={23--40},
  year={2017},
  publisher={Springer}
}

@article{wang2024survey,
  title={A survey on convex optimization for guidance and control of vehicular systems},
  author={Wang, Zhenbo},
  journal={Annual Reviews in Control},
  volume={57},
  pages={100957},
  year={2024},
  publisher={Elsevier}
}
\end{document}